\documentclass[a4paper,12pt]{amsart}

\usepackage[margin=20mm]{geometry}
\usepackage{amsfonts,amsmath,amssymb,amsthm,colonequals}
\usepackage{caption}
\usepackage{xcolor,tikz,fp}
\usetikzlibrary{calc, decorations.markings, patterns, arrows.meta}
\tikzset{                  
  n/.style={anchor=north,shift={(.1,.2)}}
}

\definecolor{gblue}{HTML}{4285f4}   
\definecolor{ggreen}{HTML}{0f9d58}  
\definecolor{gyellow}{HTML}{f4b400} 
\definecolor{gred}{HTML}{db4437}    

\usepackage{hyperref}

\usepackage{zref-clever}
\zcsetup{cap}

\newtheorem{theorem}{Theorem}
\newtheorem{proposition}[theorem]{Proposition}
\newtheorem{lemma}[theorem]{Lemma}

\theoremstyle{definition}
\newtheorem{definition}[theorem]{Definition}
\newtheorem{remark}[theorem]{Remark}
\newtheorem{example}[theorem]{Example}

\zcRefTypeSetup{observation}{
  name-sg-ab = {observation},  
  Name-sg-ab = {Observation},  
  name-pl-ab = {observations}, 
  Name-pl-ab = {Observations},  
  abbrev
}

\newcommand\codiag{\raisebox{8pt}{\rotatebox{180}{$\triangle$}}}
\newcommand\Conf{\textup{Conf}}
\newcommand\diag{\rotatebox{0}{$\triangle$}}
\newcommand\Exp{\textup{Exp}}
\newcommand\id{\textup{id}}

\newcommand\Mb{\textup{Mb}}

\newcommand\Ran{\textup{Ran}}

\newcommand\Z{\mathbf{Z}}

\newcommand\define[1]{{\textbf{#1}}}
\newcommand\piccap[1]{\parbox{2.5cm}{\centering\vphantom{Ap}\textit{#1}}}

\renewcommand\epsilon\varepsilon

\colorlet{a}{white}
\colorlet{b}{black!80}
\colorlet{ab}{a!50!b}
\newcommand\hspacing{.4} 
\newcommand\hbb{1.4} 

\title[Constructions of contractibility]{Constructions of contractibility in finite subset spaces}
\author{J\=anis Lazovskis}
\address{Institute of Mathematics, University of Aberdeen, Aberdeen AB24 3FX, United Kingdom}
\curraddr{Institute of Clinical and Preventive Medicine, University of Latvia, Riga LV-1586, Latvia}
\email{janis.lazovskis@lu.lv}
\date{\today}

\begin{document} 

\begin{abstract}
The space of all finite non-empty subsets of a topological space $X$, also known as the Ran space of $X$, is weakly contractible for $X$ path connected.
We consider subspaces $\Ran_{\leqslant n}(X)$ of the Ran space given by all subsets of $X$ of size at most $n$, and their first homotopy groups.
These groups are known to be trivial for $n\geqslant 3$, though the proofs are not constructive.
We draw an explicit contraction in $\Ran_{\leqslant 3}(S^1)$ of the generator of $\pi_1(S^1)$, which we then use to conlude that the induced maps $\pi_1(\Ran_{\leqslant n}(X)) \to \pi_1(\Ran_{\leqslant 2n+1}(X))$ and $\pi_1(\Ran_{\leqslant n}(X)) \to \pi_1(\Ran_{\leqslant n+2}(X))$ are trivial for all positive integers $n$.
Pathological maps with infinitely many strata changes in $\Ran_{\leqslant n}(X)$ are resolved via locally contractible neighbourhoods and exit paths.
From this we conclude the triviality of $\pi_1(\Ran_{\leqslant n}(X))$, for all $n\geqslant 3$.
\end{abstract}

\maketitle

\section{Introduction}
Let $X$ be a topological space.
Finite subsets of $X$, also known as configurations, and the spaces they define, also known as finite subset spaces, provide a way to encode the geometry and topology of a space.
They appear in various fields of mathematics \cite{MR2605307}, from braid groups \cite{MR375281} to higher algebra \cite{dag4} and factorization homology \cite{MR2058353}.
Withholding the choice of topology for a moment, the set of all finite subsets of $X$ of size $n$ is known as the (unordered) \define{configuration space} $\Conf_n(X)$ of $X$, and we call the collection of all configuration spaces for positive $n$ as the \define{Ran space} $\Ran(X)$.
Restricting to configurations of size at most $n$, we get the space $\Ran_{\leqslant n}(X)$, which we also call the \define{Ran space}, when the difference is clear from context.
The empty set in $X$, as a configuration of size 0, together with $\Ran(X)$ defines the \define{exponential} $\Exp(X)$ of $X$.
The naming of and symbols for these spaces are not universal, and variations with different decorations abound; the present choice in informed by the relationship to stratifications and algebraic constructions \cite{Lazovskis2019}.

\begin{figure}[hbtp!]
    \centering
    \begin{tikzpicture}
        \node at (0,0) (p) {\includegraphics{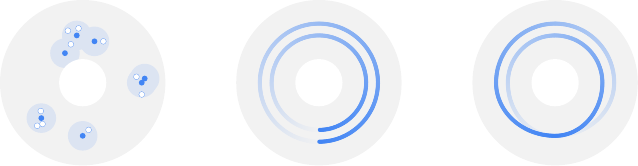}};
        \node[shift={(0,0)}] at (p.north west) {a.\vphantom{Ap}};
        \node[shift={(4,0)}] at (p.north west) {b.\vphantom{Ap}};
        \node[shift={(8,0)}] at (p.north west) {c.\vphantom{Ap}};
    \end{tikzpicture}
    \captionsetup{width=.9\linewidth}
    \caption{
    Two nearby configurations on a topological space $X$ (a.), one in a darker color the other in white, with a neighborhood emphasized around the former.
    A continuous map $S^1\to \Conf_2(X)$ composed of two ``nearby'' maps $S^1\to X$ (b.).
    A continuous map $S^1\to \Ran_{\leqslant 2}(X)$ composed of maps $S^1\to X$ going in opposite directions (c.).
    A continuous image $[0,1]\to X$ is drawn by a path going from a light $(t=0)$ to a dark color $(t=1)$.
    }
    \label{fig_rantopology}
\end{figure} 

We take the topology on $\Ran(X)$ and $\Ran_{\leqslant n}(X)$ to be the coarsest topology (that is, having as few open sets as possible) for which, given any finite collection $\{U_i\subseteq X\}_i$ of nonempty, pairwise disjoint open sets $U_i$, the set $\{$finite $U\subseteq X\ :\  U\subseteq \bigcup_iU_i,\ U\cap U_i \neq\emptyset\ \forall\ i\}$ is open.
This definition follows \cite{dag4}, and corresponds with the metric induced by the Hausdorff metric on subsets when $X$ has a metric topology.
The intuition, presented in \zcref{fig_rantopology}, is that given an element $x\in \Ran(X)$, some other $y\in \Ran(X)$ is in a ``small neighborhood'' of $x$ whenever each small neighborhood (in the topology on $X$) of each point of $x$ contains at least one point of $y$. 
A more detailed discussion of common topologies on spaces of configurations is presented in \cite{MR4871695}.

\subsection{Background}

For this note, we ask that $X$ be path-connected and Hausdorff, to avoid cumbersome notation and pathological spaces.
For any $x\in X$, we write $c_x \colon Y\to X$ for the constant map that sends everything in $Y$ to $x$, and $\id \colon X\to X$ for the identity map.
In the space of all finite nonempty subsets of $X$, there are natural \textbf{comultiplication} and \textbf{multiplication} maps
\begin{equation}
  \label{eq_bdmap}
  \Ran(X) \xrightarrow{\ \diag\ } \Ran(X)\times \Ran(X) \xrightarrow{\ \codiag\ } \Ran(X),
\end{equation}s
given by $\diag (P) = (P,P)$ and $\codiag (P,Q) = P\cup Q$, respectively. 
These naturally induce maps on the space of continuous functions $S^1\to \Ran(X)$, given by
\begin{align}
\label{eq_diag_formula}
\diag(\sigma) & = \iota_1(\sigma) \circ \iota_2(\sigma) , \\
\label{eq_codiag_formula}
\codiag(\sigma,\tau) & = \left( \begin{array}{r c l} S^1 & \to & \Ran(X) \\ t & \mapsto & \sigma(t) \cup \tau(t) \end{array} \right),
\end{align}
where $\iota_1$ is inclusion into the first factor of $\Ran(X)\times \Ran(X)$, and $\iota_2$ is inclusion into the second factor.
Composition $\circ$ of maps from $S^1$ and $[0,1]$ is taken left-to-right, for example $\sigma \circ \tau$ meaning ``$\sigma$ followed by $\tau$'', as drawing such a composition, we would draw the image of the map on the left first, then the image of the map on the right\footnote{However, we keep with the right-to-left convention when using arbitrary functions. That is, we write $g\circ f$ to denote the map that takes $t$ to $g(f(t))$.}.

\begin{remark}
In the pointed, or based Ran space $(\Ran(X),b)$, that is, the subspace of $\Ran(X)$ of all elements containing a chosen basepoint $b\in \Ran(X)$, combining the data of \zcref{eq_diag_formula,eq_codiag_formula} gives us that 
\begin{equation}
\codiag(\diag(\sigma)) 
= \codiag(\sigma \circ c_{b}, c_{b} \circ \sigma)
= (\sigma \circ c_{b}) \cup (c_{b} \circ \sigma)
= (\sigma \cup c_{b}) \circ (c_{b} \cup \sigma)
= \sigma \circ \sigma, 
\end{equation}
for any continuous $\sigma \colon S^1 \to (\Ran(X),b)$.
Here we use the set union symbol ``$\cup$'' for paths to mean $(\sigma \cup \tau)(t) = \sigma(t) \cup \tau (t)$.
Together with \zcref{eq_ran_composition}, we immediately recover one of the arguments for simple connectivity of these spaces (see \zcref{sec_related_work}).
\end{remark}

The same symbols $\diag,\codiag$ for maps are used for convenience, with the choice of domain (either $\Ran(X)$ or maps of $S^1$ into $\Ran(X)$) being clear from context.
By functoriality, the (co)multiplication map also extends to a map on homotopy groups, in which case we use $\diag_*,\codiag_*$, respectively.

These maps also make sense on bounded nonempty subsets of $X$, but there are differences.
Here and later we write $\iota \colon \Ran_{\leqslant n}(X) \hookrightarrow \Ran_{\leqslant m}(X)$ for the natural inclusion map, for any positive integers $n \leqslant m$.
Most notably, the compositions $\codiag \circ \diag$ produce maps
\begin{equation}
\label{eq_ran_composition}
\codiag \circ \diag \simeq \id_{\Ran(X)} \colon \Ran(X) \to \Ran(X),
\hspace{1cm}
\codiag \circ \diag \simeq \iota \colon \Ran_{\leqslant n}(X) \to \Ran_{\leqslant 2n}(X),
\end{equation}
an observation that we revisit in \zcref{sec_related_work}.
Here and later, the symbol $\simeq$ between maps denotes homotopy equivalence.

To prove \zcref{prop_inf2fin}, we put more structure on our $X$, asking that it be a CW-complex with a particular stratification, using notation from Lurie \cite{dag4}.

\subsection{Related work}
\label{sec_related_work}

As with most sources on this topic, we begin with an overview of the historical development, with spaces and maps often appearing under different names and symbols. 
To keep our attention focused, we mainly consider work on the triviality of inclusion maps and on particular constructions that witness this triviality.

Beilinson and Drinfeld \cite[Section 3.4.1]{MR2058353}, among others, show that all the homotopy groups of $\Ran(X)$ are trivial when $X$ is path-connected.
Their proof does not provide an explicit homotopy from $S^n \to \Ran(X)$ to the trivial morphism $*\to \Ran(X)$, and their argument on $\Ran(X)$ can not be directly extended to $\Ran_{\leqslant n}(X)$. 
Indeed, the argument \cite{MR2058353,MR4871695} for triviality relies on the composition $\codiag\circ \diag$ being the identity, whereas on the Ran spaces $\Ran_{\leqslant n}(X)$, the codiagonal map has target $\Ran_{\leqslant 2n}(X)$, so the analogous composition is the inclusion from $\Ran_{\leqslant n}(X)$ into $\Ran_{\leqslant 2n}(X)$, which is not the identity map.

Such inclusions are further considered by others, in particular Handel \cite{MR1823966}, who applies the composition of \zcref{eq_ran_composition} to prove that the maps $\pi_k(\Ran_{\leqslant n}(X)) \rightarrow \pi_k(\Ran_{\leqslant 2n+1}(X))$ induced by inclusions are trivial, for all integers $k\geqslant 1$ and all $n\geqslant 0$.
Handel builds on the same result using the based Ran space.
The result for based spaces relies on having a natural construction for the wedge sum of two based Ran spaces, which are wedged at the basepoint $\{b\}$.
The map on the non-based Ran spaces is then factored through the based versions by adjoining the basepoint, hence the ``$+1$'' in the index of the result.

The results of Tuffley \cite{MR2105772} give that $\pi_1(\Ran_{\leqslant n}(X)) = 0$ for $X$ a CW-complex and for all $n \geqslant 3$, though the general statement also considers the connectedness of $X$ in higher dimensions and higher homotopy groups, and applies to infinite complexes.
This statement uses special covers of $X$ and the Hurewicz theorem, relying on a choice of cells of $X$, which are not explicitly determined.

Kallel and Sjerve \cite{hha/1296138520} work with simplicial complexes $X$ to also prove $\pi_1(\Ran_{\leqslant n}(X)) = 0$ for all $n \geqslant 3$, by constructing $\Ran_{\leqslant n}(X)$ as a quotient of the $n$-fold symmetric product of $X$.
They use the van Kampen theorem to get the existence of a particular element of $\pi_1(X\times X)$ which is a witness for the triviality of $\Ran_{\leqslant 3}(X)$, followed by an inductive argument for higher indices. 

Symmetric products also appear in the work of F\'elix and Tanr\'e \cite{MR2648705}, who prove similar results about the triviality on the level of homotopy of inclusion maps $\Ran_{\leqslant n}(X) \to \Ran_{\leqslant m}(X)$.
The indices $m,n$ are related to the Lusternik--Schnierlmann category of the CW-complex $X$, and they also consider based Ran spaces.
Similarly to the previously mentioned work, the triviality results follow from existence proofs, without explicit constructions.

\subsection{Contribution}
Our main contribution is a constructive argument for the triviality of any continuous map $\sigma \colon S^1 \to \Ran_{\leqslant n}(X)$, for $X$ an arbitrary path-connected, locally-contractible, Hausdorff topological space.
The construction of this triviality, via homotopy equivalences, is shown to be contained within $\Ran_{\leqslant 2n+1}(X)$ in one case, and within $\Ran_{\leqslant n+2}(X)$ with a different approach, the latter of which uses the naturally stratified structure of Ran spaces.
We provide visual impressions of this construction, both as a continuous sequence of subsets of $\Ran_{\leqslant 3}(S^1)$ (\zcref{fig_main_construx}) and as a path of a single element through $\Ran_{\leqslant 3}(S^1)$ (\zcref{fig_main_construx_mb}).
This construction, along with a resolution of pathological cases, allows us to conclude that $\Ran_{\leqslant n}(X)$ has trivial first homotopy group, for all $n\geqslant 3$.

\section{Explicit constructions}
\label{sec_explicit_constructions}

We begin with our central result, which motivates all further discussions.

\begin{lemma}
\label{lem_main_construx}
The map on homotopy groups $\pi_1(\Ran_{\leqslant 1}(S^1)) \rightarrow \pi_1(\Ran_{\leqslant 3}(S^1))$ induced by inclusion of spaces is trivial.
\end{lemma}

\begin{proof}
Let $\gamma \colon S^1 \to S^1$ be the identity map $\id$.
Consider the maps
\begin{align}
H_a \colon [0,1] \times S^1 & \to S^1, \\
(s,t) & \mapsto \{\min((1+s)t,1),\max(0,(1+s)t-s)\},
\nonumber\\[12pt]
H_b \colon [0,1] \times S^1 & \to S^1, \\
(s,t) & \mapsto \{\min((1+s)t,1),\max(0,\min((1+s)t,2s-(1+s)t)),0\},
\nonumber\\[12pt]
H_c \colon [0,1] \times S^1 & \to S^1, \\
(s,t) & \mapsto \{\min((2-2s)t,(2s-2)(t-1)),0\}.
\nonumber
\end{align}
These homotopies are visually described for several values of $s$ in (respectively) parts a., b., c. of \zcref{fig_main_construx}.
By construction they are continuous in both variables, and satisfy
\begin{align}
H_a(0,t) & = \gamma(t),
\\
H_b(0,t) & = H_a(1,\textstyle \frac t2) = H_a(1,\textstyle \frac {3t}2),
\\
H_c(0,t) & = H_b(1,t),
\\
H_c(1,t) & = 0,
\end{align}
for all $t\in S^1$.
Combine them together in a map $H\colon [0,1]\times S^1 \to \Ran_{\leqslant 3}(S^1)$, defined by 
\begin{equation}
\label{eq_main_homotopy}
H(s,t) = \begin{cases}
H_a(3s,t) & 0 \leqslant s < 1/3, \\
\codiag(\diag(H_b(3s-1,\ \cdot\ )))(t) & 1/3 \leqslant s < 2/3 , \\
\codiag(\diag(H_c(3s-2,\ \cdot\ )))(t) & 2/3 \leqslant s \leqslant 1.
\end{cases}
\end{equation}
This map is also continuous by construction, and satisfies $H(0,t) = \gamma(t)$ and $H(1,t) = 0$, for all $t\in [0,1]$. 
Hence $H$ makes the diagram
\begin{equation}
\label{eq_main_construx}
\begin{tikzpicture}[baseline=1.5cm,xscale=2.5,yscale=1.5]
    \node (s1a) at (0,1) {$S^1$\vphantom{Ap}};
    \node (s1b) at (1,1) {$S^1$\vphantom{Ap}};
    \node (s1c) at (0,0) {$S^1$\vphantom{Ap}};
    \node (ran) at (1,0) {$\Ran_{\leqslant 3}(S^1)$\vphantom{Ap}};
    \draw[->] (s1a) to node[above] {$\gamma=\id$} (s1b);
    \draw[->] (s1a) to node[left] {0} (s1c);
    \draw[{Hooks[right]}->] (s1c) to node[above] {$\iota$} (ran);
    \draw[{Hooks[right]}->] (s1b) to node[right] {$\iota$} (ran);
    \node[anchor=west] at (s1b.east) {$=\ \Ran_{\leqslant 1}(S^1)$\vphantom{Ap}};
    \node[anchor=east] at (s1c.west) {$\Ran_{\leqslant 1}(S^1)\ =$\vphantom{Ap}};
\end{tikzpicture}
\end{equation}
commute up to homotopy.
By functoriality of $\pi_1$, the diagram
\begin{equation}
\label{eq_main_construx}
\begin{tikzpicture}[baseline=1.5cm,xscale=4,yscale=1.5]
    \node (s1a) at (0,1) {$\pi_1(S^1)$\vphantom{Ap}};
    \node (s1b) at (1,1) {$\pi_1(\Ran_{\leqslant 1}(S^1))$\vphantom{Ap}};
    \node (s1c) at (0,0) {$\pi_1(\Ran_{\leqslant 1}(S^1))$\vphantom{Ap}};
    \node (ran) at (1,0) {$\pi_1(\Ran_{\leqslant 3}(S^1))$\vphantom{Ap}};
    \draw[->] (s1a) to node[above] {$\id$} (s1b);
    \draw[->] (s1a) to node[left] {0} (s1c);
    \draw[{Hooks[right]}->] (s1c) to node[above] {$\iota_*$} (ran);
    \draw[{Hooks[right]}->] (s1b) to node[right] {$\iota_*$} (ran);
\end{tikzpicture}
\end{equation}
also commutes.
The identity is not the zero map, so $\iota_*$ must be the zero map, as desired.
\end{proof}

Note that the homotopies presented in \zcref{lem_main_construx} are all basepoint-preserving.

\begin{figure}[hbtp!]
    \centering
    \begin{tikzpicture}
        \node[anchor=west] (s1) at (0,10) {\includegraphics{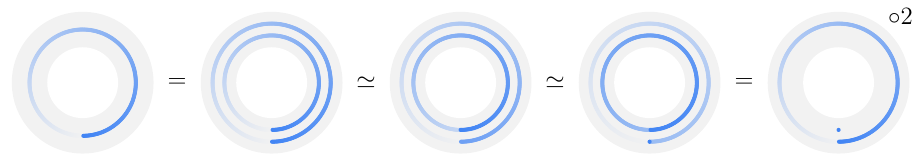}};
        \node[anchor=west] (s2) at (0,5) {\includegraphics{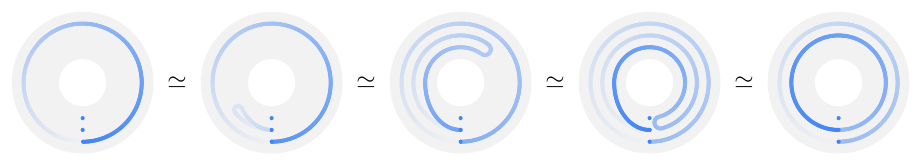}};
        \node[anchor=west] (s3) at (0,0) {\includegraphics{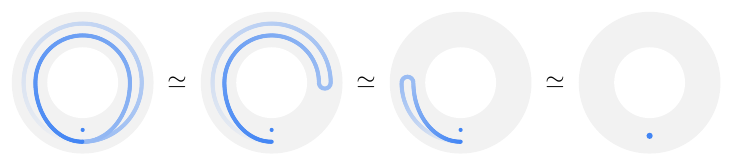}};
        \node[anchor=north west] at (s1.north west) {\vphantom{Ap}a.};
        \node[anchor=north west] at (s2.north west) {\vphantom{Ap}b.};
        \node[anchor=north west] at (s3.north west) {\vphantom{Ap}c.};
        \node[n] at ($(s1.south west)+(0:\hbb)$) {\piccap{loop $\gamma$}};
        \node[n] at ($(s1.south west)+(0:3*\hbb+\hspacing)$) {\piccap{another copy of $\gamma$}};
        \node[n] at ($(s1.south west)+(0:5*\hbb+2*\hspacing)$) {\piccap{hold (opposite) end of each copy}};
        \node[n] at ($(s1.south west)+(0:7*\hbb+3*\hspacing)$) {\piccap{one point moves while other is fixed}};
        \node[n] at ($(s1.south west)+(0:9*\hbb+4*\hspacing)$) {\piccap{composition of $\gamma$ and constant loop with itself}};
        \node[n] at ($(s2.south west)+(0:\hbb)$) {\piccap{loop $\gamma \cup c_{\gamma(0)}$}};
        \node[n] at ($(s2.south west)+(0:3*\hbb+\hspacing)$) {\piccap{push out trivial loop from $c_{\gamma(0)}$ and hold end of $\gamma$}};
        \node[n] at ($(s2.south west)+(0:5*\hbb+2*\hspacing)$) {\piccap{push out for longer and hold end of $\gamma$ longer}};
        \node[n] at ($(s2.south west)+(0:7*\hbb+3*\hspacing)$) {\piccap{trivial loop meets $\gamma(0)$ when $\gamma$ stops moving}};
        \node[n] at ($(s2.south west)+(0:9*\hbb+4*\hspacing)$) {\piccap{identify two loops}};
        \node[n] at ($(s3.south west)+(0:\hbb)$) {\piccap{loop $(\gamma \circ \gamma^{-1}) \cup c_{\gamma(0)}$}};
        \node[n] at ($(s3.south west)+(0:3*\hbb+\hspacing)$) {\piccap{pull back end of $\gamma$ and start of $\gamma^{-1}$}};
        \node[n] at ($(s3.south west)+(0:5*\hbb+2*\hspacing)$) {\piccap{keep shrinking}};
        \node[n] at ($(s3.south west)+(0:7*\hbb+3*\hspacing)$) {\piccap{trivial loop}};
    \end{tikzpicture}
    \captionsetup{width=.9\linewidth}
    \caption{
    Path homotopies for the proof of \zcref{lem_main_construx}, drawn in $S^1\times [0,1] \simeq S^1$, to better demonstrate the individual steps.
    Radial separation between loops is used to only visually distinguish the different paths, and should be considered as all at the same radius.
    }
    \label{fig_main_construx}
\end{figure} 

\begin{remark}
\label{rem_ranloops}
The natural inclusion $\iota \colon \Ran_{\leqslant 1}(S^1) \to \Ran_{\leqslant 3}(S^1)$ used in \zcref{lem_main_construx} is known to have image the trefoil knot, by Shchepin \cite{mostovoy}, as the boundary of the non-orientable surface 
\[
\raisebox{-1.2cm}{\includegraphics{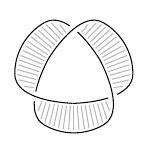}}
= \ \Ran_{\leqslant 2}(S^1) \ \supseteq \ \Ran_{\leqslant 1}(S^1)
\ = \raisebox{-1.2cm}{\includegraphics{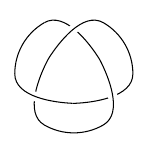}}.
\]
This confirms the observation of Ghys \cite{ghys} that $\Ran_{\leqslant 1}(S^2) \cong \Mb$, as the Mobius band is homeomorphic to the Mobius band with two more twists.
As a result, we may interpret the constructions of several loops on $S^1$ from \zcref{fig_main_construx} as instead constructions of single loop on the Mobius band.
For a better sense of the relationship, we note that 
\[
\raisebox{-1cm}{\includegraphics{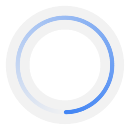}}
= \raisebox{-1cm}{\includegraphics{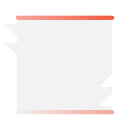}}
, \ \ \ \ 
\raisebox{-1cm}{\includegraphics{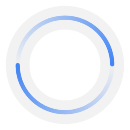}}
= \raisebox{-1cm}{\includegraphics{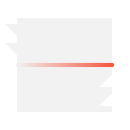}}
, \ \ \ \ 
\raisebox{-1cm}{\includegraphics{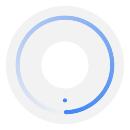}}
= \raisebox{-1cm}{\includegraphics{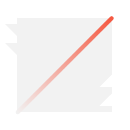}},
\]
where the Mobius band is drawn as a square with left and right sides identified by a reflection in the vertical axis.
A full reinterpretation of the homotopies presented \zcref{fig_main_construx}, as they appear on the Mobius band, is given in \zcref{fig_main_construx_mb}.
\end{remark}

\begin{figure}[hbtp!]
    \centering
    \begin{tikzpicture}
        \node[anchor=west] (s1) at (0,5) {\includegraphics{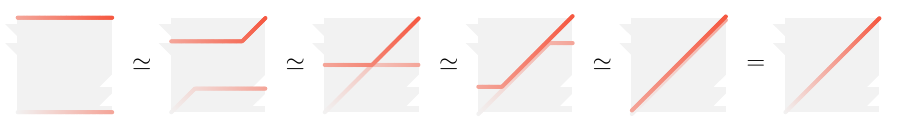}};
        \node[anchor=west] (s2) at (0,2.5) {\includegraphics{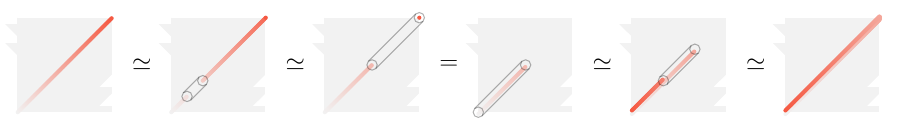}};
        \node[anchor=west] (s3) at (0,0) {\includegraphics{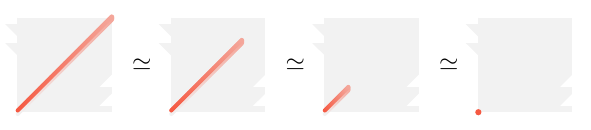}};
        \node[anchor=north] at (s1.north east) {$\circ 2$};
        \node[anchor=north east] at (s1.north west) {\vphantom{Ap}a.};
        \node[anchor=north east] at (s2.north west) {\vphantom{Ap}b.};
        \node[anchor=north east] at (s3.north west) {\vphantom{Ap}c.};
    \end{tikzpicture}
    \captionsetup{width=.9\linewidth}
    \caption{
    The homotopies of \zcref{fig_main_construx} as they appear in the Ran space, with corresponding parts a., b., c. showing the same homotopies in both figures.
    Part of the homotopy in b. is not on $\Mb$, and that part is demarcated by solid lines, where the projection to two of the three points would lie.
    }
    \label{fig_main_construx_mb}
\end{figure} 

These constructions will be generalized to arbitrary loops in \zcref{sec_contracting_arbitrary_loops}.

\subsection{Types of points and pathological examples}

To extend the construction of \zcref{lem_main_construx} to arbitrary loops, we define instances at which the cardinality of a continuous $\sigma \colon S^1 \to \Ran_{\leqslant m}(X)$ changes.

\begin{definition}
\label{def_branch_merge_regular}
Let $\sigma \colon S^1 \to \Ran(X)$ be continuous, and $p\in X$, with $p\in \sigma(t)$ for some $t\in [0,1]$.
\begin{itemize}
    \item The point $p$ is a \define{branch point} of $\sigma$ at $t<1$ if for every open neighborhood $U\subseteq X$ of $p$, there exists $\epsilon>0$ satisfying $|U\cap \sigma(t+\epsilon)| >1$.
    \item Analogously, $p$ is a \define{merge point} of $\sigma$ at $t>0$ if for every open neighborhood $U\subseteq X$ of $p$, there exists $\epsilon>0$ satisfying $|U\cap \sigma(t-\epsilon)| >1$.
    \item Conversely, $p$ is a \define{regular point} of $\sigma$ at $t$ if for every open neighborhood $U\subseteq X$ of $p$, there exists $\epsilon>0$ satisfying $|U \cap \sigma(t \pm \epsilon')| = 1$, for every $0 \leqslant \epsilon' < \epsilon$.
\end{itemize}
The values $\epsilon,\epsilon'$ are considered only when they make sense for the domain $[0,1]$ of $\sigma$.
\end{definition}

It is immediate that when $X$ is Hausdorff, every point $p\in \sigma(t)$ is a branch point, a merge point, or a regular point at $t$.
Such a $p$ may be simultaneously a branch point and a merge point at $t$.
If $p$ is a regular point at $t$, it cannot be a branch point nor a merge point at $t$. 

\begin{example}
\label{ex_nonhausdorff}
Separability of points in $X$ is  key to the previous observations.
For example, let $X$ be the line with two origins, and $\sigma \colon S^1 \to \Ran(X)$ a continuous map with $\sigma(t_0)$ containing both origins for exactly one $t_0$ in some non-trivial subinterval of $[0,1]$.
In this case, neither of the origins meet the conditions for being a branch point, merge point, or regular point at $t_0$.
\end{example}

\begin{example}
\label{obs_main_counterex}
For every $k\in \Z_{>0}$, there exists $\sigma \colon S^1 \to \Ran_{\leqslant 3}(X)$ that does not factor through $X^\ell$, for each $\ell=1,\dots,k$.
Moreover, there exists $\sigma \colon S^1 \to \Ran_{\leqslant 3}(X)$ that does not factor through $X^k$ for all $k\in \Z_{>0}$.
These examples are presented visually in \zcref{fig_main_counterex}a. and \zcref{fig_main_counterex}b., respectively.
\end{example}

\begin{figure}[hbtp!]
    \centering
    \includegraphics{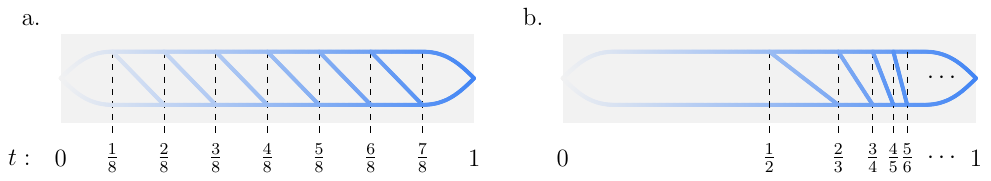}
    \captionsetup{width=.9\linewidth}
    \caption{
    Visual presentations of the counterexamples from \zcref{obs_main_counterex}, for $k=7$ on the left (a.) and the first four branch and merge points drawn on the right (b.).
    Here $X = S^1\times [0,1]$, and $t=0$ is identified with $t=1$ for both (a.) and (b.).
    }
    \label{fig_main_counterex}
\end{figure}

\subsection{Supporting observations}

To prove our claims in the following section, we make some initial observations about paths and loops.
The first statement applies the construction of \zcref{lem_main_construx} to a path in $X$ considered as a loop in $\Ran_{\leqslant 3}(X)$, after multiplying by the constant path at both its endpoints.

\begin{lemma}
\label{lem_assum1}
For every continuous $\sigma \colon [0,1] \to X$, the composition $S^1 \to \Ran_{\leqslant 3}(X) \to \Ran_{\leqslant 4}(X)$ given by $\iota \circ \codiag(\sigma,c_{\{\sigma(0),\sigma(1)\}})$ is homotopic to $c_{\{\sigma(0),\sigma(1)\}}$.
\end{lemma}

\begin{proof}
This follows from the same steps as in the proof of \zcref{lem_main_construx}, performing (homotopies analogous to) $H_b$ and $H_c$, separating the origin of the circle $S^1$ into two distinct points.
Indeed, we see that $ \codiag(\sigma,c_{\{\sigma(0),\sigma(1)\}})$ is
\[
\includegraphics{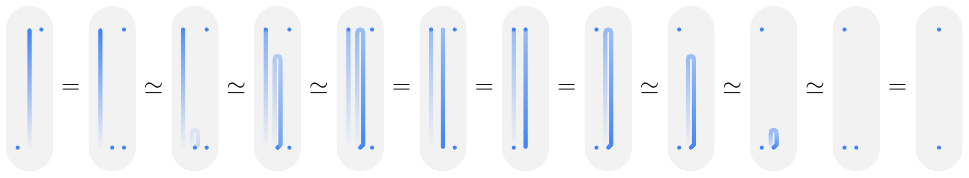}\ ,
\]
which is just $c_{\{\sigma(0),\sigma(1)\}}$. 
Here horizontal separation was used to only visually distinguish different paths, and should be considered as all at the same horizontal value (analogously to \zcref{fig_main_construx}).
The first image is considered as a map $S^1 \to \Ran_{\leqslant 3}(X)$, and the last one is the constant map $c_{\{\sigma(0),\sigma(1)\}}\colon S^1 \to \Ran_{\leqslant 2}(X)$.
These maps make the diagram
\begin{equation}
\label{eq_assum1}
\begin{tikzpicture}[baseline=1.6cm,xscale=5,yscale=1.6]
    \node (int) at (0,2) {$[0,1]$\vphantom{Ap}};
    \node (x) at (1,2) {$X$\vphantom{Ap}};
    \node (s1) at (0,1) {$S^1$\vphantom{Ap}};
    \node (nx1) at (1,1) {$\Ran_{\leqslant 3}(X)$\vphantom{Ap}};
    \node (nx2) at (0,0) {$\Ran_{\leqslant 2}(X)$\vphantom{Ap}};
    \node (2nx) at (1,0) {$\Ran_{\leqslant 4}(X)$\vphantom{Ap}};
    \draw[->] (int) to node[above] {$\sigma$} (x);
    \draw[->] (int) to node[left] {$0 \sim 1$} (s1);
    \draw[->] (x) to node[right] {$\codiag(\ \cdot\ , \{\sigma(0),\sigma(1)\})$} (nx1);
    \draw[->] (s1) to node[above] {$\codiag(\sigma,c_{\{\sigma(0),\sigma(1)\}}$} (nx1);
    \draw[->] (s1) to node[left] {$c_{\{\sigma(0),\sigma(1)\}}$} (nx2);
    \draw[{Hooks[right]}->] (nx2) to node[above] {$\iota$} (2nx);
    \draw[{Hooks[right]}->] (nx1) to node[right] {$\iota$} (2nx);
\end{tikzpicture}
\end{equation}
commute up to homotopy in the bottom rectangle, and commute without qualification in the top rectangle.
The desired statement is commutativity of the bottom rectangle.
\end{proof}

\begin{lemma}
\label{lem_assum2}
For every continuous $\sigma \colon S^1 \to \Ran_{\leqslant n}(X)$ with finitely many branch points and merge points, there exist continuous $\gamma_i \colon S^1 \to \Ran_{\leqslant 3}(X)$ for $i=1,\dots,k$ so that $\sigma(t) = \bigcup_{i=1}^k \{\gamma_i(t)\}$.
\end{lemma}

\begin{proof}
Let $\sigma$ be as described. 
Then for every regular point $p$ of $\sigma$ at $t$, there exists a unique interval $[t_0,t_1]\subseteq [0,1]$ containing $t$ in its interior and a continuous function $\sigma_p\colon [t_0,t_1]\to X$ satisfying
\begin{itemize}
    \item $\sigma_p(t') \in \sigma(t')$ for all $t'\in [t_0,t_1]$;
    \item $\sigma_p(t')$ is a regular point of $\sigma$ at $t'$, for every $t'\in (t_0,t_1)$;
    \item either $\sigma_p(t_0)$ is a branch point at $t_0$, a merge point at $t_0$, or $t_0=0$; and
    \item either $\sigma_p(t_1)$ is a branch point at $t_1$, a merge point at $t_1$, or $t_1=1$.
\end{itemize}
The interval $[t_0,t_1]$ exists by \zcref{def_branch_merge_regular}, and is unique because it is the maximal element of the poset of intervals $[t_0,t_1]$ satisfying the first two conditions.
Moreover, this construction of $\sigma_p$ defines an equivalence relation on the subset of $X\times [0,1]$ of pairs $(p,t)$, where $p$ is a regular point of $\sigma$ at $t$. 
As there are finitely many branch points and merge points, there are finitely many equivalence classes, and so finitely many pairs of intervals $[t_0,t_1]$ and images $\sigma_p([t_0,t_1])$.

Suppose there are $k>0$ such pairs.
For each such pair of $[t_{i,0},t_{i,1}]$ and $\sigma_{p_i}([t_{i,0},t_{i,1}])$, let $\gamma_i\colon [0,1]\to X$ be a continuous path with $\gamma_i(t)\in \sigma(t)$ for all $t\in [0,1]$, as well as $\gamma_i(t) = \sigma_{p_i}(t)$ for all $t\in [t_{i,0},t_{i,1}]$. 
These $\gamma_i$ are not necessarily unique.
Their existence may be assumed, or they can be constructed by concatenating the $\sigma_{p_i}$ with (possibly more than one) appropriate $\sigma_q$, so that, for example, $t_{i,0} = t_{i,j}$ and $\sigma_p(t_{i,0}) = \sigma_q(t_{j,1})$.
As there are finitely many branch points and merge points, and as $X$ is Hausdorff, this construction is well-defined.

The domain of $\gamma_i$ is $[0,1]$, and it may be that $\gamma_i(0) \neq \gamma_i(1)$.
To guarantee this does happen, redefine $\gamma_i$ as $\codiag(\gamma_i,c_{\{\gamma_i(0),\gamma_i(1)\}})$, which makes $\gamma_i$ into a map $S^1\to \Ran_{\leqslant 3}(X)$.
We now have $\bigcup_{i=1}^k \{\gamma_i(t)\} = \sigma(t)$ for all $t\in [0,1]$, as desired.
\end{proof}

The claim of \zcref{lem_assum2} says that for every continuous $\sigma \colon S^1 \to \Ran_{\leqslant n}(X)$ with finitely many branch points and merge points, there exist $\gamma_1,\dots,\gamma_k$ that ensure the commutativity of
\begin{equation}
\label{commsq1}
\begin{tikzpicture}[baseline=2cm,xscale=2,yscale=2]
    \node (s1) at (0,1) {$S^1$\vphantom{Ap}};
    \node (ran1) at (2,1) {$\Ran_{\leqslant n}(X)$\vphantom{Ap}};
    \node (ran3) at (1,0) {$\Ran_{\leqslant 3}(X)^k$\vphantom{Ap}\ .};
    \draw[->] (s1) to node[above] {$\sigma$} (ran1);
    \draw[->] (s1) to node[auto,swap] {$(\gamma_1,\dots,\gamma_k)$} (ran3);
    \draw[->] (ran3) to node[auto,swap] {$\codiag$} (ran1);
\end{tikzpicture}
\end{equation}

\begin{remark}
\label{rem_basepoint}
The maps $\gamma_i$ in the claim of \zcref{lem_assum2} have $\Ran_{\leqslant 3}(X)$ as their range, but this can be improved to $\Ran_{\leqslant 2}(X)$, if we allow relationships up to homotopy.
That is, for every continuous $\sigma \colon S^1 \to \Ran_{\leqslant n}(X)$ and any $b\in X$, there exists a continuous $\hat \sigma \colon S^1 \to \Ran_{\leqslant n}(X)$ homotopic to $\sigma$ with $\hat\sigma(0) = \hat\sigma(1) = \{b\}$.
Applying the construction of the proof of \zcref{lem_assum2} to this $\hat\sigma$ will produce initial $\gamma_i$ with domain $S^1$, that is, $\gamma_i(0) = \gamma_i(1)$.
\end{remark}

\section{Contracting arbitrary loops}
\label{sec_contracting_arbitrary_loops}

In this section we generalize \zcref{lem_main_construx} in two different ways.
Our first claim generalizes \zcref{lem_main_construx} by considering its constructions on the input of an arbitrary loop $\sigma$.

\begin{proposition}
\label{thm_main1}
For every continuous $\sigma \colon S^1\to \Ran_{\leqslant n}(X)$, the composition $S^1 \xrightarrow{\ \sigma\ } \Ran_{\leqslant n}(X) \xrightarrow{\ i\ } \Ran_{\leqslant 2n+1}(X)$ is homotopic to the constant map at $\sigma(0)$, for all positive integers $n$.
\end{proposition}

\begin{proof}
Let $\sigma \colon S^1 \to \Ran_{\leqslant n}(X)$ be as in the claim.
Consider the following maps, generalized from the proof of \zcref{lem_main_construx}:
\begin{align}
H_a \colon [0,1] \times S^1 & \to \Ran_{\leqslant 2n}(X), \\
(s,t) & \mapsto \codiag\big(\sigma(\min((1+s)t,1)),\sigma(\max(0,(1+s)t-s))\big),
\nonumber\\[12pt]
H_b \colon [0,1] \times S^1 & \to \Ran_{\leqslant 2n+1}(X), \\
(s,t) & \mapsto \codiag\big(\sigma(\min((1+s)t,1)),\sigma(\max(0,\min((1+s)t,2s-(1+s)t))),c_{\{b\}}\big),
\nonumber\\[12pt]
H_c \colon [0,1] \times S^1 & \to \Ran_{\leqslant 2n}(X), \\
(s,t) & \mapsto \codiag\big(\sigma(\min((2-2s)t,(2s-2)(t-1))), c_{\{b\}}\big).
\nonumber
\end{align}
Here we have slightly abused notation by giving $\codiag$ three arguments in $H_b$, to mean the repeated application of $\codiag$.
As before, these maps are continuous by construction, and satisfy
\begin{align}
H_a(0,t) & = \sigma(t),
\\ 
H_b(0,t) & = H_a(1,\textstyle \frac t2) = H_a(1,\textstyle \frac {3t}2), 
\\
H_c(0,t) & = H_b(1,t),
\\ 
H_c(1,t) & = c_{\sigma(0)}(t),
\end{align}
for all $t\in S^1$.
Combine these together in precisely the same way as \zcref{eq_main_homotopy}, to get a continuous map $H_X\colon [0,1]\times S^1 \to \Ran_{\leqslant 2n+1}(X)$. 
This map will then satisfy $H_X(0,t) = \sigma(t)$ and $H_X(1,t) = 0$, for all $t\in [0,1]$, proving the claim.
\end{proof}

Next we generalize \zcref{lem_main_construx} by applying its constructions to each individual factor of $\sigma$, when factoring through $X^n$.
This allows us to construct a homotopy in a smaller Ran space, but only when we have finitely many branch points and merger points.

\begin{proposition}
\label{thm_main2}
For every continuous $\sigma \colon S^1\to \Ran_{\leqslant n}(X)$ with finitely many branch points and merge points, the composition $S^1 \xrightarrow{\ \sigma\ } \Ran_{\leqslant n}(X) \xrightarrow{\ i\ } \Ran_{\leqslant n+2}(X)$ is homotopic to the trivial map, for all positive integers $n$.
\end{proposition}

\begin{proof}
As $\sigma$ has finitely many branch points and merge points, by \zcref{lem_assum2} there exist continuous $\gamma_i \colon S^1 \to \Ran_{\leqslant 3}(X)$ for $i=1,\dots,k$ so that $\sigma(t) = \bigcup_{i=1}^k \{\gamma_i(t)\}$ for all $t$.
Reparametrize these individual $\gamma_i$ so that all except one are fixed at their endpoints for every $t$.
That is, let $H_\gamma \colon [0,1] \times S^1 \to \Ran_{\leqslant n}(X)$ be given by $H_\gamma(s,t) = \bigcup_{i=1}^k \{\hat \gamma_i(s,t)\}$, for
\begin{equation}
\label{thm_main2_reparametrizers}
\hat\gamma_i(s,t) = \begin{cases}
\gamma_i((1-s)t) & \text{\ if\ } t \leqslant \frac{i-1}k, \\
\gamma_i(t-ts+s+(kt-i)s) & \text{\ if\ } \frac{i-1}k < t \leqslant \frac{i}k, \\
\gamma_i(t-ts+s) & \text{\ if\ } t > \frac{i}k,
\end{cases}
\hspace{1cm}
\begin{tikzpicture}[baseline=.6cm,xscale=.9]
  \draw[->] (0,-.2) -- (0,1.7);
  \draw[->] (-.2,0) -- (3.2,0) node[right] {$t$};
  \foreach \x\l in {0/0, .7/{i-1}, 1.1/{i}, 3/{k}}{
    \draw[dashed] (\x,1.5) -- (\x,-.1) node[below,scale=.8] {$\frac{\l}{k}$};
  }
  \node[scale=.8,anchor=east] at (-.2,0) {$\gamma_i(0)$};
  \node[scale=.8,anchor=east] at (-.2,1.5) {$\gamma_i(1)$};
  \draw (0,1.5)--(-.2,1.5);
  \draw[line width=1.5pt,gblue] (0,0) -- (3,1.5);
  \draw[line width=1.5pt,gred] (0,0) 
    -- (.7,0)
    -- (1.1,1.5) 
    to node[above,gred] (l) {$\hat\gamma_i(1,t)$} (3,1.5);
  \node[gblue,yshift=-1.45cm] at (l) {$\hat\gamma_i(0,t)$};
\end{tikzpicture}
\end{equation}
for all $s\in [0,1]$. 
This map $H_\gamma$ is continuous in both variables, and satisfies
\begin{align}
\label{main2_h1}
H_\gamma(s,0) & = \sigma(0),
& H_\gamma(0,t) & = \sigma(t),
\\
\label{main2_h2}
H_\gamma(1,\textstyle \frac ik) & \subseteq \sigma(0), 
&
H_\gamma(1,\textstyle\frac{t+i-1}{k}) & \subseteq \codiag(\gamma_i(t),c_{\sigma(0)})
\end{align}
for all $i=1,\dots,k$ and for all $s,t\in [0,1]$.
That is, the map $H_\gamma$ asserts that the diagram
\begin{equation}
\label{eq_comtri1}
\begin{tikzpicture}[baseline=1.9cm,xscale=1.8,yscale=2]
    \node (s1) at (0,1) {$S^1$\vphantom{Ap}};
    \node (ran1) at (2,1) {$\Ran_{\leqslant n}(X)$\vphantom{Ap}};
    \node (ran11) at (4,1) {$\Ran_{\leqslant n+1}(X)$\vphantom{Ap}};
    \node (ran31) at (2,0) {$\Ran_{\leqslant 3}(X)^k$\vphantom{Ap}};
    \draw[->] (s1) to node[above] {$\sigma$} (ran1);
    \draw[->] (s1) to node[auto,swap] {$(\hat\gamma_1,\dots,\hat\gamma_k)$} (ran31);
    \draw[{Hooks[right]}->] (ran1) to node[above] {$\iota$} (ran11);
    \draw[->] (ran31) to node[auto,swap] {$\codiag$} (ran11);
\end{tikzpicture}
\end{equation}
commutes up to homotopy, in particular providing that the image of $\codiag$ lies in $\Ran_{\leqslant n+1}(X)$, given the $k$-fold product of the $\hat\gamma_i(S^1)$ as input.

Applying \zcref{lem_assum1} to each individual $\hat\gamma_i$ (appropriately reparametrized), we get that the composition $S^1 \to \Ran_{\leqslant 3}(X)^k \xrightarrow{\ \iota\ } \Ran_{\leqslant 4}(X)^k$ is homotopic to the product of the constant maps at the basepoints of the $\hat\gamma_i$. 
Combining this with the inclusion $\iota$ of the source and target of $\codiag$ from \zcref{eq_comtri1} into a larger Ran space, we get that the diagram
\begin{equation}
\label{eq_comtri2}
\begin{tikzpicture}[baseline=1.9cm,xscale=1.8,yscale=2]
    \node (s1) at (0,1) {$S^1$\vphantom{Ap}};
    \node (ran1) at (2,1) {$\Ran_{\leqslant n}(X)$\vphantom{Ap}};
    \node (ran11) at (4,1) {$\Ran_{\leqslant n+1}(X)$\vphantom{Ap}};
    \node (ran12) at (6,1) {$\Ran_{\leqslant n+2}(X)$\vphantom{Ap}};
    \node (ran31) at (2,0) {$\Ran_{\leqslant 3}(X)^k$\vphantom{Ap}};
    \node (ran32) at (4,0) {$\Ran_{\leqslant 4}(X)^k$\vphantom{Ap}};
    \node (ran2) at (2,-1) {$\Ran_{\leqslant 2}(X)^k$\vphantom{Ap}};
    \draw[->] (s1) to node[above] {$\sigma$} (ran1);
    \draw[->] (s1) to node[auto,pos=.7] {$(\hat\gamma_1,\dots,\hat\gamma_k)$} (ran31);
    \draw[->] (s1) to node[auto,swap] {$(c_{\hat\gamma_1(0)},\dots,c_{\hat\gamma_k(0)})$} (ran2);
    \draw[{Hooks[right]}->] (ran1) to node[above] {$\iota$} (ran11);
    \draw[{Hooks[right]}->] (ran11) to node[above] {$\iota$} (ran12);
    \draw[{Hooks[right]}->] (ran31) to node[above] {$\iota$} (ran32);
    \draw[{Hooks[right]}->] (ran2) to node[auto,swap] {$\iota$} (ran32);
    \draw[->] (ran32) to node[auto,swap] {$\codiag$} (ran12);
\end{tikzpicture}
\end{equation}
commutes up to homotopy.
Here again we used the properties of $H_\gamma$ from \zcref{main2_h2} for the source and target of $\codiag$.
The bottom rectangle is the same as the bottom rectangle of \zcref{eq_assum1}, applied to each of the $k$ factors.
\end{proof}

As we have previously seen in \zcref{obs_main_counterex}, there are maps from $S^1$ into $\Ran_{\leqslant n}(X)$ which can not be factored through a finite number of copies of $X$. 
To deal with this situation, we reduce to the finite case with local homotopies, using the structure of a conically stratified space.
For this section, we ask that $X$ is a CW-complex, as we need local contractibility for \zcref{prop_inf2fin}.

\begin{definition}
Let $X$ be a topological space and $(A,\geqslant)$ a partially ordered set.
Considering $A$ as a topological space with base of topology the upwards directed sets $A_{\geqslant x} \colonequals \{y\in A\ :\ y \geqslant x\}$, a \define{stratification} of $X$ is a continuous map $X\to A$, which makes $X$ into a \define{stratified} space.
\end{definition}

We recall for any stratifications $f\colon X\to A$ and $g\colon Y\to B$, a stratified map is a pair of functions $\phi\colon X\to Y$ and $\psi\colon A\to B$, satisfying $g\circ \phi = \psi\circ f$.
We are interested in \define{conical} stratifications, for which every element of $x\in X$ has an open neighbourhood, called a \define{conical neighbourhood}, that looks like a cone of some stratified space, with the induced stratification matching the restriction from $X$.
We skip over the details, and direct the reader to \cite{dag4,Lazovskis2022} for broader and contextually relevant exposition.

\begin{example}
The spaces $\Ran(X)$ and $\Ran_{\leqslant n}(X)$ are conically stratified spaces, as witnessed by the point-counting function $\Ran(X)\to \Z_{>0}$ given by $P \mapsto |P|$.
Similarly, the $n$-simplex $\Delta^n$ has its geometric realization $|\Delta^n|$ conically stratified by a map into $[n]$, given either by 
\begin{itemize}
    \item $(t_0,\dots,t_i \neq 0,0,\dots,0) \mapsto i$, or
    \item $(0,\dots,0,t_i \neq 0,\dots,t_n) \mapsto n-i$.
\end{itemize}
A stratified map $|\Delta^n|\to X$ respecting the former stratification is called an \define{exit path}, and respecting the latter stratification is called an \define{entry path}.
\end{example}

For the following claim, we talk about \define{accumulation points}.
In this context, we consider an accumulation point to be a branch point or merge point $p\in \Ran_{\leqslant n}(X)$ of $\sigma$, so that for every $\epsilon>0$, the open $\epsilon$-ball $B(p,\epsilon)\subseteq \Ran_{\leqslant n}(X)$ contains infinitely many branch points or merge points of $\sigma$.

\begin{figure}[hbtp!]
    \centering
    \begin{tikzpicture}
        \node at (0,0) (p) {\includegraphics{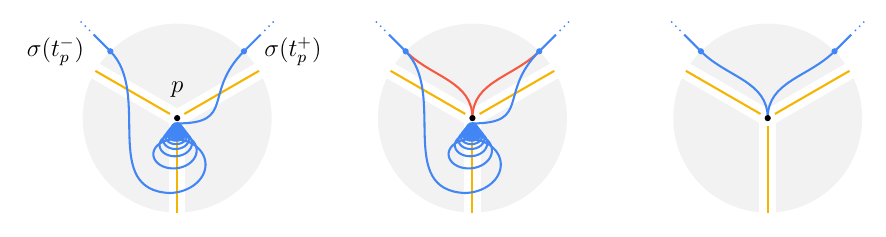}};
        \node[shift={(-14,.7)}] at (p.south east) {a.\vphantom{Ap}};
        \node[shift={(-9,.7)}] at (p.south east) {b.\vphantom{Ap}};
        \node[shift={(-4,.7)}] at (p.south east) {c.\vphantom{Ap}};
    \end{tikzpicture}
    \captionsetup{width=.9\linewidth}
    \caption{
    A visual description of the content of \zcref{prop_inf2fin}, with an example of $p$ being a merge point that is an accumulation point.
    The 0-dimensional strata are drawn in black, 1-dimensional strata in yellow, and 2-dimensional strata in gray.
    In the conical neighbourhood (a.) of $p$, we can find values $t_p^- < t_p^+$ for which the image of $\sigma$ on the subinterval $[t_p^-,t_p^+]$ is completely within this neighbourhood.
    Replacing the image of $\sigma$ on this subinterval with a composition of an entry path (b., left) and an exit path (b., right), we get a path $\hat\sigma$ homotopic to $\sigma$ that has no accumulation point in the same neighbourhood (c.).
    }
    \label{fig_inf2fin}
\end{figure} 

\begin{proposition}
\label{prop_inf2fin}
For every continuous $\sigma \colon S^1 \to \Ran_{\leqslant n}(X)$, there exists a continuous $\hat \sigma \colon S^1 \to \Ran_{\leqslant n}(X)$ homotopic to $\sigma$ with finitely many branch points and merge points.
\end{proposition}

\begin{proof}
For every $p$ that is an accumulation point of $\sigma$ at $t_p\in [0,1]$, let $\epsilon_p >0$ be such that the open ball $B(p,\epsilon_p)\subseteq \Ran_{\leqslant n}(X)$ is contained within the local conical neighbourhood of $p$ and is contractible.
Such a contractible open ball exists, as $X$ is a CW-complex (and so is locally contractible), and so $\Ran_{\leqslant n}(X)$ is a CW-complex \cite[Proposition 3.1]{hha/1296138520}. 

Let $t_p^-< t_p < t_p^+\in [0,1]$ be such that $\sigma(t)\in B(p,\epsilon_p)$ for every $t\in (t_p^-,t_p^+)$. 
Such values exist as $\sigma$ is continuous, and by the sequential characterization of continuity, as when $t\to t_p$, it must also be that $\sigma(t) \to \sigma(t_p)$. 

If $\sigma(t_p^-) = \sigma(t_p^+)$, then as $B(p,\epsilon_p)$ is contractible, the map $\hat \sigma$ defined by
\begin{equation}
t \mapsto \begin{cases}
\sigma(t) & \text{\ if\ } t \not\in [t_p^-,t_p^+], \\
\sigma(t_p^-) = \sigma(t_p^+) & \text{\ if\ } t \in (t_p^-,t_p^+),
\end{cases}
\end{equation}
is homotopic to $\sigma$, and does not have $p$ as an accumulation point.
If $\sigma(t_p^-) \neq \sigma(t_p^+)$, then define the entry and exit paths $\tau_p^-,\tau_p^-\colon |\Delta^1|\to B(p,\epsilon_p)$ by
\begin{align}
\tau_p^-(0) & = \sigma(t_p^-),
&
\tau_p^-(1) & = \sigma(p),
\\
\tau_p^+(0) & = \sigma(p),
&
\tau_p^+(1) & = \sigma(t_p^+).
\end{align}
For $f\colon \Ran_{\leqslant n}(X) \to \Z_{>0}$ the stratifying map of $\Ran_{\leqslant n}(X)$, by construction $f(\sigma(t_p^-)) \geqslant f(\sigma(p))$ and $f(\sigma(t_p^+)) \geqslant f(\sigma(p))$. 
This observation, and that $B(p,\epsilon_p)$ is within a conical neighbourhood of $p$, is enough to conclude that such paths exist, as we are simply asking for paths that stay within the strata of $\sigma(t_p^-)$ and $\sigma(t_p^+)$ except for their terminal and initial points, respectively. 
In this case, the map $\hat \sigma$ is defined by
\begin{equation}
t \mapsto \begin{cases}
\sigma(t) & \text{\ if\ } t \not\in [t_p^-,t_p^+], \\
\tau_p^-\left(\frac{2(t-t_p^-)}{t_p^+-t_p^-}\right) & \text{\ if\ } t \in \left(t_p^-,\frac{t_p^++t_p^-}{2}\right], \\
\tau_p^+\left(\frac{2(t-\frac{t_p^++t_p^-}{2})}{t_p^+-t_p^-}\right) & \text{\ if\ } t \in \left(\frac{t_p^++t_p^-}{2},t_p^+\right].
\end{cases}
\end{equation}
As before, $\hat \sigma$ is homotopic to $\sigma$, and does not have $p$ as an accumulation point.
A visual description of the relationship between $\sigma$ and $\hat\sigma$ is given by example in \zcref{fig_inf2fin}b.

Doing this for all accumulation points will give us a continuous map, which we still call $\hat\sigma$, that is homotopic to $\sigma$ and does not have any accumulation points.
As the domain of $\hat \sigma$ is $S^1$ and as $\hat \sigma$ is continuous, this is enough to conclude that $\hat\sigma$ has finitely many branch points and merge points. 
\end{proof}

We now have enough to combine and generalize several of our previous claims.

\begin{theorem}
\label{thm_main3}
Let $X$ be a path-connected, Hausdorff topological space, and $n\in \Z_{>0}$.
\begin{enumerate}
    \item The map on homotopy groups $\pi_1(\Ran_{\leqslant n}(X)) \rightarrow \pi_1(\Ran_{\leqslant 2n+1}(X))$ induced by inclusion of spaces is trivial.
    \item If $X$ is also locally contractible, the map on homotopy groups $\pi_1(\Ran_{\leqslant n}(X)) \rightarrow \pi_1(\Ran_{\leqslant n+2}(X))$ induced by inclusion of spaces is trivial.
\end{enumerate}
\end{theorem}

\begin{proof}
Let $[\sigma] \in \pi_1(\Ran_{\leqslant n}(X))$, with $\sigma\colon S^1 \to \Ran_{\leqslant n}(X)$ a representative of $[\sigma]$.
For the first statement, \zcref{thm_main1} and the functoriality of $\pi_1$ give that the diagrams 
\begin{equation}
\label{eq_thm_main3a}
\begin{tikzpicture}[baseline=1.5cm,xscale=3,yscale=1.5]
    \node (s1) at (0,1) {$S^1$\vphantom{Ap}};
    \node (nx1) at (1,1) {$\Ran_{\leqslant n}(X)$\vphantom{Ap}};
    \node (nx2) at (0,0) {$\Ran_{\leqslant n}(X)$\vphantom{Ap}};
    \node (2nx) at (1,0) {$\Ran_{\leqslant 2n+1}(X)$,\vphantom{Ap}};
    \draw[->] (s1) to node[above] {$\sigma$} (nx1);
    \draw[->] (s1) to node[left] {$c_{\sigma(0)}$} (nx2);
    \draw[{Hooks[right]}->] (nx2) to node[above] {$\iota$} (2nx);
    \draw[{Hooks[right]}->] (nx1) to node[right] {$\iota$} (2nx);
\end{tikzpicture}
\hspace{.7cm}
\begin{tikzpicture}[baseline=1.5cm,xscale=4,yscale=1.5]
    \node (s1) at (0,1) {$\pi_1(S^1)$\vphantom{Ap}};
    \node (nx1) at (1,1) {$\pi_1(\Ran_{\leqslant n}(X))$\vphantom{Ap}};
    \node (nx2) at (0,0) {$\pi_1(\Ran_{\leqslant n}(X))$\vphantom{Ap}};
    \node (2nx) at (1,0) {$\pi_1(\Ran_{\leqslant 2n+1}(X))$\vphantom{Ap}};
    \draw[->] (s1) to node[above] {$\sigma_*$} (nx1);
    \draw[->] (s1) to node[left] {0} (nx2);
    \draw[{Hooks[right]}->] (nx2) to node[above] {$\iota_*$} (2nx);
    \draw[{Hooks[right]}->] (nx1) to node[right] {$\iota_*$} (2nx);
\end{tikzpicture}
\end{equation}
commute up to homotopy.
Indeed, as $X$ is path-connected, the constant map $c_{\sigma(0)}$ and the trivial map (considered as the constant map at a single point of $X$) are in the same class in $\pi_1(\Ran_{\leqslant n}(X))$ (compare with \zcref{rem_basepoint}).

It follows that either $\sigma_*$ or $\iota_*$ are the trivial map.
As $\pi_1(S^1)$ is not trivial and $\sigma_*$ is arbitrary, either $\pi_1(\Ran_{\leqslant n}(X))$ is trivial or $\iota_*$ is the trivial map.
If $\iota_*$ is trivial, then the first statement is proven, and if\footnote{We know this to be true for $n\geqslant 3$ by \cite{MR2105772}, but it being true is not necessary here.} $\pi_1(\Ran_{\leqslant n}(X)) = 0$, then $\iota_*$ must also be trivial, and so still the first statement holds. 

For the second statement, by \zcref{prop_inf2fin} and as $X$ is Hausdorff and locally contractible, we may assume that $\sigma$ has finitely many branch points and merge points. 
Then \zcref{thm_main2}, in particular \zcref{eq_comtri2}, and the functoriality of $\pi_1$ give that the diagrams
\begin{equation}
\label{eq_thm_main3b}
\begin{tikzpicture}[baseline=1.5cm,xscale=3,yscale=1.5]
    \node (s1) at (0,1) {$S^1$\vphantom{Ap}};
    \node (nx1) at (1,1) {$\Ran_{\leqslant n}(X)$\vphantom{Ap}};
    \node (nx2) at (0,0) {$\Ran_{\leqslant 2}(X)^k$\vphantom{Ap}};
    \node (2nx) at (1,0) {$\Ran_{\leqslant n+2}(X)$,\vphantom{Ap}};
    \draw[->] (s1) to node[above] {$\sigma$} (nx1);
    \draw[->] (s1) to node[left] {$(c_{\hat\gamma_1(0)},\dots,c_{\hat\gamma_k(0)})$} (nx2);
    \draw[{Hooks[right]}->] (nx2) to node[above] {$\iota$} (2nx);
    \draw[{Hooks[right]}->] (nx1) to node[right] {$\iota$} (2nx);
\end{tikzpicture}
\hspace{.7cm}
\begin{tikzpicture}[baseline=1.5cm,xscale=4,yscale=1.5]
    \node (s1) at (0,1) {$\pi_1(S^1)$\vphantom{Ap}};
    \node (nx1) at (1,1) {$\pi_1(\Ran_{\leqslant n}(X))$\vphantom{Ap}};
    \node (nx2) at (0,0) {$\pi_1(\Ran_{\leqslant 2}(X)^k)$\vphantom{Ap}};
    \node (2nx) at (1,0) {$\pi_1(\Ran_{\leqslant n+2}(X))$\vphantom{Ap}};
    \draw[->] (s1) to node[above] {$\sigma_*$} (nx1);
    \draw[->] (s1) to node[left] {0} (nx2);
    \draw[{Hooks[right]}->] (nx2) to node[above] {$\iota_*$} (2nx);
    \draw[{Hooks[right]}->] (nx1) to node[right] {$\iota_*$} (2nx);
\end{tikzpicture}
\end{equation}
commute up to homotopy.
The conclusion that $\iota_*$ is the trivial map follows as before.
\end{proof}

This same method, with a slight modification earlier based on \zcref{rem_basepoint}, provides an explicit construction of a homotopy from any loop in $\Ran_{\leqslant n}(X)$ to the trivial loop, in which the homotopy is always within $\Ran_{\leqslant n}(X)$ or $\Ran_{\leqslant 3}(X)$, whichever of $n$ or $3$ is larger.

\begin{theorem}
\label{thm_main4}
Let $X$ be a path-connected, Hausdorff, locally contractible topological space.
For every positive integer $n\geqslant 3$, $\pi_1(\Ran_{\leqslant n}(X)) = 0$.
\end{theorem}

\begin{proof}
Taking a representative $\sigma$ of $[\sigma] \in \pi_1(\Ran_{\leqslant n}(X))$ with finitely many branch points and merge points, and for which $\sigma(0) = \{b\}$, \zcref{lem_assum2} provides us with $k$ maps $\gamma_i \colon S^1 \to \Ran_{\leqslant 2}(X)$ all based at $b$, satisfying $\sigma(t) = \bigcup_{i=1}^k \{\gamma_i(t)\}$.
The proof of \zcref{lem_assum1} with this setup provides us with a homotopy $H_\gamma$ that satisfies $|H_\gamma(1,t)|\in \{1,2\}$ for all $t\in [0,1]$, as $\sigma(0)$ is a single point.
As such, in the commutative diagram analogous to \zcref{eq_comtri1}, we only need to include in a larger Ran space if $n$ is less than 3.
As a result, the commutative diagram of \zcref{eq_comtri2} becomes
\begin{equation}
\label{eq_comtri3}
\begin{tikzpicture}[baseline=1.9cm,xscale=1.8,yscale=2]
    \node (s1) at (0,1) {$S^1$\vphantom{Ap}};
    \node (ran1) at (2,1) {$\Ran_{\leqslant n}(X)$\vphantom{Ap}};
    \node (ran12) at (5.2,1) {$\Ran_{\leqslant \max(3,n)}(X)$\vphantom{Ap}\ ,};
    \node (ran31) at (2,0) {$\Ran_{\leqslant 2}(X)^k$\vphantom{Ap}};
    \node (ran32) at (3.6,0) {$\Ran_{\leqslant 3}(X)^k$\vphantom{Ap}};
    \node (ran2) at (2,-1) {$\Ran_{\leqslant 1}(X)^k$\vphantom{Ap}};
    \draw[->] (s1) to node[above] {$\sigma$} (ran1);
    \draw[->] (s1) to node[auto,pos=.7] {$(\hat\gamma_1,\dots,\hat\gamma_k)$} (ran31);
    \draw[->] (s1) to node[auto,swap] {$(c_{b},\dots,c_{b})$} (ran2);
    \draw[{Hooks[right]}->] (ran1) to node[above] {$\iota$} (ran12);
    \draw[{Hooks[right]}->] (ran31) to node[above] {$\iota$} (ran32);
    \draw[{Hooks[right]}->] (ran2) to node[auto,swap] {$\iota$} (ran32);
    \draw[->] (ran32) to node[auto,swap] {$\codiag$} (ran12);
\end{tikzpicture}
\end{equation}
which provides us with the constant map $c_b\in [\sigma]$ homotopic to $\sigma$.
\end{proof}

Note that \zcref{thm_main4} does not hold for $n=1,2$, so this is the best that we can get.
Indeed, $\Ran_{\leqslant 1}(X) = \Conf_1(X) \cong X$, which does not have trivial first homotopy group, if $X$ does not have trivial first homotopy group.
Moreover for $X=S^1$ \cite{MR1998017}, we have $\Ran_{\leqslant 2}(S^1) \cong (S^1\times S^1)/((x,y)\sim (y,x))$ which is the Mobius band, which also does not have trivial first homotopy group.

\section{Concluding remarks}

In this note we have described explicit constructions to contract loops in bounded finite subset spaces of path-connected, Hausdorff (and where necessary, locally contractible) topological spaces.
Simple and visual constructions for maps $S^1 \to S^1$ were extended to arbitrary continuous maps $S^1 \to \Ran_{\leqslant n}(X)$, with the methods demonstrating triviality of the inclusion map $\Ran_{\leqslant n}(X) \hookrightarrow \Ran_{\leqslant 2n+1}(X)$, when considered on equivalence classes in the first homotopy groups of these spaces.

Due to pathological behavior in these finite subset spaces (\zcref{obs_main_counterex}, \zcref{fig_inf2fin}), the broader insight of these simple, visual constructions is initially limited.
By restricting to the class of loops whose cardinality (as a subset of $X$) changes finitely many times, we also demonstrated that the inclusion map  $\Ran_{\leqslant n}(X) \hookrightarrow \Ran_{\leqslant n+2}(X)$ is trivial in the first homotopy groups, by decomposing loops into their constituent parts.
Applying local homotopies to avoid pathological behavior allowed us, up to homotopy, to extend this result to arbitrary loops, and to conclude the triviality of a large class of homotopy groups.

\begin{remark}
We mention two directions for further development of these ideas.
\begin{enumerate}
    \item \textbf{Bounding the number of factors:} An upper bound was provided for the value of $k$ in \zcref{lem_assum2}, when decomposing a continuous map $\sigma \colon S^1 \to \Ran_{\leqslant n}(X)$ as a product of $k$ individual maps from $S^1$. 
    Similar to Morse theory and persistent homology, it seems that determining a pairing between merge points and branch points of $\sigma$ would allow us to find a smaller lower bound, perhaps even precisely determine the minimal $k$.
    \item \textbf{Extending to higher dimensions:} Our work is exclusively devoted to the 1-dimensional sphere $S^1$, as it is amenable to drawings and visual constructions.
    As the triviality of higher homotopy groups is also known, it seems reasonable that the constructions presented here could be generalized to higher homotopy groups.
\end{enumerate}
\end{remark}

\section*{Acknowledgements}
This note is based in part on the PhD thesis of the author \cite[Section 4.3]{Lazovskis2019}.
The author thanks Ambrose Yim for a motivating discussion prompting renewed effort, Sadok Kallel for providing more background on related work, Sylvain Douteau for helpful conversations and suggesting the interpretation of \zcref{rem_ranloops}, and an anonymous reviewer for identifying the nuance described in \zcref{ex_nonhausdorff}.
The author is supported by the Latvian Council of Science (LZP) 1.1.1.9 Research application No 1.1.1.9/LZP/1/24/125 of the Activity ``Post-doctoral Research" ``Efficient topological signatures for representation learning in medical imaging". 

\bibliographystyle{alpha}
\bibliography{references}

@book {MR2058353,
    AUTHOR = {Beilinson, Alexander and Drinfeld, Vladimir},
     TITLE = {Chiral algebras},
    SERIES = {American Mathematical Society Colloquium Publications},
    VOLUME = {51},
 PUBLISHER = {American Mathematical Society, Providence, RI},
      YEAR = {2004},
     PAGES = {vi+375},
      ISBN = {0-8218-3528-9},
   MRCLASS = {17Bxx (14F43)},
  MRNUMBER = {2058353},
MRREVIEWER = {Francisco J. Plaza Mart\'\i n},
       DOI = {10.1090/coll/051},
       URL = {https://doi.org/10.1090/coll/051},
}

@phdthesis{Lazovskis2019,
    author  = "J\=anis Lazovskis",
    title   = "{Stability of Universal Constructions for Persistent Homology}",
    year    = "2019",
    address = {Chicago, IL, USA},
    school  = {University of Illinois at Chicago},
    type    = {Ph{D} thesis},
    doi     = "10.25417/uic.12481691.v1"
}

@article{Lazovskis2022,
    author  =   {Lazovskis, J{\={a}}nis},
    title   =   {Stratifications on the Ran Space},
    journal =   {Order},
    year    =   {2022},
    month   =   {04},
    day     =   {01},
    volume  =   {39},
    number  =   {1},
    pages   =   {143-158},
    issn    =   {1572-9273},
    doi     =   {10.1007/s11083-021-09568-1},
    url     =   {https://doi.org/10.1007/s11083-021-09568-1}
}

@article {MR1823966,
    AUTHOR = {Handel, David},
     TITLE = {Some homotopy properties of spaces of finite subsets of
              topological spaces},
   JOURNAL = {Houston J. Math.},
  FJOURNAL = {Houston Journal of Mathematics},
    VOLUME = {26},
      YEAR = {2000},
    NUMBER = {4},
     PAGES = {747--764},
      ISSN = {0362-1588},
   MRCLASS = {55Q52 (54B15 55P65 55R80)},
  MRNUMBER = {1823966},
MRREVIEWER = {Leonard\ R.\ Rubin},
}

@article {MR4871695,
    AUTHOR = {Cepek, Anna and Lejay, Damien},
     TITLE = {On the topologies of the exponential},
   JOURNAL = {Port. Math.},
  FJOURNAL = {Portugaliae Mathematica. A Journal of the Portuguese
              Mathematical Society},
    VOLUME = {82},
      YEAR = {2025},
    NUMBER = {1-2},
     PAGES = {1--30},
      ISSN = {0032-5155,1662-2758},
   MRCLASS = {55R80 (54A10 57N80 57R56)},
  MRNUMBER = {4871695},
       DOI = {10.4171/pm/2131},
       URL = {https://doi.org/10.4171/pm/2131},
}

@ARTICLE{dag4,
       author = {{Lurie}, Jacob},
        title = "{Derived Algebraic Geometry VI: E\_k Algebras}",
      journal = {arXiv e-prints},
         year = 2009,
        month = 10,
          eid = {arXiv:0911.0018},
        pages = {arXiv:0911.0018},
          doi = {10.48550/arXiv.0911.0018},
archivePrefix = {arXiv},
       eprint = {0911.0018},
 primaryClass = {math.AT},
       adsurl = {https://ui.adsabs.harvard.edu/abs/2009arXiv0911.0018L}
}

@book {MR375281,
    AUTHOR = {Birman, Joan S.},
     TITLE = {Braids, links, and mapping class groups},
    SERIES = {Annals of Mathematics Studies},
    VOLUME = {No. 82},
 PUBLISHER = {Princeton University Press, Princeton, NJ; University of Tokyo
              Press, Tokyo},
      YEAR = {1974},
     PAGES = {ix+228},
   MRCLASS = {55A25},
  MRNUMBER = {375281},
MRREVIEWER = {Wilbur\ Whitten},
}

@incollection {MR2605307,
    AUTHOR = {Cohen, Frederick R.},
     TITLE = {Introduction to configuration spaces and their applications},
 BOOKTITLE = {Braids},
    SERIES = {Lect. Notes Ser. Inst. Math. Sci. Natl. Univ. Singap.},
    VOLUME = {19},
     PAGES = {183--261},
 PUBLISHER = {World Sci. Publ., Hackensack, NJ},
      YEAR = {2010},
      ISBN = {978-981-4291-40-8; 981-4291-40-4},
   MRCLASS = {55R80 (20F36 55U10 57M07)},
  MRNUMBER = {2605307},
MRREVIEWER = {Don\ Shimamoto},
       DOI = {10.1142/9789814291415\_0003},
       URL = {https://doi.org/10.1142/9789814291415_0003},
}

@article {MR1998017,
    AUTHOR = {Tuffley, Christopher},
     TITLE = {Finite subset spaces of {$S^1$}},
   JOURNAL = {Algebr. Geom. Topol.},
  FJOURNAL = {Algebraic \& Geometric Topology},
    VOLUME = {2},
      YEAR = {2002},
     PAGES = {1119--1145},
      ISSN = {1472-2747,1472-2739},
   MRCLASS = {54B20 (55Q52 55R80 57M25)},
  MRNUMBER = {1998017},
MRREVIEWER = {Jacob\ Mostovoy},
       DOI = {10.2140/agt.2002.2.1119},
       URL = {https://doi.org/10.2140/agt.2002.2.1119},
}

@article {MR2105772,
    AUTHOR = {Tuffley, Christopher},
     TITLE = {Connectivity of finite subset spaces of cell complexes},
   JOURNAL = {Pacific J. Math.},
  FJOURNAL = {Pacific Journal of Mathematics},
    VOLUME = {217},
      YEAR = {2004},
    NUMBER = {1},
     PAGES = {175--179},
      ISSN = {0030-8730,1945-5844},
   MRCLASS = {55P99},
  MRNUMBER = {2105772},
       DOI = {10.2140/pjm.2004.217.175},
       URL = {https://doi.org/10.2140/pjm.2004.217.175},
}

@article{mostovoy,
 ISSN = {00029890, 19300972},
 URL = {http://www.jstor.org/stable/4145248},
 author = {Jacob Mostovoy},
 journal = {The American Mathematical Monthly},
 number = {4},
 pages = {357--360},
 publisher = {Mathematical Association of America},
 title = {Lattices in $\mathbb{C}$ and Finite Subsets of a Circle},
 urldate = {2026-02-10},
 volume = {111},
 year = {2004}
}

@incollection {MR2648705,
    AUTHOR = {F\'elix, Yves and Tanr\'e, Daniel},
     TITLE = {Rational homotopy of symmetric products and spaces of finite
              subsets},
 BOOKTITLE = {Homotopy theory of function spaces and related topics},
    SERIES = {Contemp. Math.},
    VOLUME = {519},
     PAGES = {77--92},
 PUBLISHER = {Amer. Math. Soc., Providence, RI},
      YEAR = {2010},
      ISBN = {978-0-8218-4929-3},
   MRCLASS = {55P62 (55M30 55S15)},
  MRNUMBER = {2648705},
       DOI = {10.1090/conm/519/10233},
       URL = {https://doi.org/10.1090/conm/519/10233},
}

@article{hha/1296138520,
author = {Sadok Kallel and Denis Sjerve},
title = {{Remarks on finite subset spaces}},
volume = {11},
journal = {Homology, Homotopy and Applications},
number = {2},
publisher = {International Press of Boston},
pages = {229 -- 250},
year = {2009},
}

@article{ghys,
author={Ghys, Etienne},
title={PROLONGEMENTS DES DIFF{\'E}OMORPHISMES DE LA SPH{\`E}RE},
journal={L'Enseignement Math{\'e}matique},
year={1991},
publisher={Fondation L'Enseignement Math{\'e}matique},
volume={37},
number={1-2},
pages={45-59},
issn={0013-8584}
}

\end{document}